\documentclass[12pt]{amsart}
\usepackage{verbatim}
\usepackage{amssymb}
\usepackage{upref}
\usepackage[all]{xy}
\usepackage[usenames,dvipsnames]{color}
\usepackage{url}
\usepackage[normalem]{ulem}

\allowdisplaybreaks

\newcommand{\sknote}[1]{{${}^*$\marginpar{\tiny{{\bf sk:} #1}}}}

\newtheorem{thm}{Theorem}[section]
\newtheorem*{thm*}{Theorem}
\newtheorem{lem}[thm]{Lemma}
\newtheorem{cor}[thm]{Corollary}
\newtheorem{prop}[thm]{Proposition}

\theoremstyle{definition}
\newtheorem{defn}[thm]{Definition}
\newtheorem{q}[thm]{Question}
\newtheorem{qs}[thm]{Questions}

\newtheorem{ex}[thm]{Example}

\newtheorem*{notn*}{Notation}

\newtheorem*{hyp*}{Hypothesis}

\newtheorem{rem}[thm]{Remark}
\newtheorem*{rem*}{Remark}

\numberwithin{equation}{section}

\newcommand{\thmref}[1]{Theorem~\textup{\ref{#1}}}

\newcommand{\midtext}[1]{\quad\text{#1}\quad}

\renewcommand{\and}{\midtext{and}}

\renewcommand{\)}{\textup)}

\newcommand{\C}{\mathbb C}

\newcommand{\CC}{\mathcal C}

\newcommand{\KK}{\mathcal K}

\renewcommand{\AA}{\mathcal A}

\renewcommand{\d}{\delta}

\newcommand{\p}{\phi}

\newcommand{\s}{\sigma}

\renewcommand{\epsilon}{\varepsilon}

\DeclareMathOperator{\aut}{Aut}

\DeclareMathOperator{\ad}{Ad}

\DeclareMathOperator{\mor}{Mor}

\DeclareMathOperator{\maxi}{Max}

\DeclareMathOperator{\nor}{Nor}

\DeclareMathOperator{\fix}{Fix}

\DeclareMathOperator{\inc}{Inc}

\DeclareMathOperator*{\spn}{span}
\DeclareMathOperator*{\clspn}{\overline{\spn}}

\newcommand{\id}{\text{\textup{id}}}

\newcommand{\inv}{^{-1}}

\newcommand{\iso}{\overset{\cong}{\longrightarrow}}
\renewcommand{\bar}{\overline}
\newcommand{\what}{\widehat}
\newcommand{\wilde}{\widetilde}

\newcommand{\rt}{\textup{rt}}

\newcommand{\cs}%
{\ensuremath{\mathbf{C^*}}}
\newcommand{\csnd}%
{\ensuremath{\cs\!\!_\mathbf{nd}}}
\newcommand{\coact}%
{\ensuremath{\mathbf{C^*coact}}}
\newcommand{\coactnd}%
{\ensuremath{\coact_\mathbf{nd}}}
\newcommand{\coactn}%
{\ensuremath{\coact^\mathbf{n}}}
\newcommand{\coactnnd}%
{\ensuremath{\coactn_\mathbf{nd}}}
\newcommand{\coactm}%
{\ensuremath{\coact^\mathbf{m}}}
\newcommand{\coactmnd}%
{\ensuremath{\coactm_\mathbf{nd}}}

\newcommand{\act}{\ensuremath{\mathbf{C^*act}}}

\newcommand{\dn}{\downarrow}

\newcommand{\RCP}{\textup{RCP}}
\newcommand{\CP}{\textup{CP}}
\newcommand{\Id}{\text{\textup{Id}}}

\newcommand{\cp}{\CP}
\newcommand{\rcp}{\RCP}

\renewcommand{\act}{\ensuremath{\AA(G)}}
\renewcommand{\coact}{\ensuremath{\CC(G)}}
\newcommand{\ncoact}{\ensuremath{\CC^n(G)}}
\newcommand{\mcoact}{\ensuremath{\CC^m(G)}}

\newcommand{\sk}{\marginpar{¥}rginpar{\tiny{\textbf{sk}}}}
\renewcommand{\sk}{\relax}
\renewcommand{\sknote}[1]{\relax}

\begin{document}

\title{Categorical perspectives in noncommutative duality}

\author{S.~Kaliszewski}
\author{John Quigg}

\thanks{These notes were written to accompany a mini-course 
given by the authors at the summer school on 
\emph{C*-algebras and their interplay with dynamical systems}
held at the Sophus Lie Conference Center
in Nordfjordeid, Norway, in June 2010.
The summer school was supported by 
the Research Council of Norway, and by NordForsk.}

\date{June 10, 2010}

\begin{abstract}
Noncommutative duality for $C^*$-dynamical systems 
is a vast generalization of Pontryagin duality 
for locally compact abelian groups.  
In this series of lectures,
we give an introduction to the categorical aspects of this duality,
focusing primarily on Landstad duality for actions and coactions
of locally compact groups.
\end{abstract}

\maketitle

%----------------------------------------------------------------

%\input{lecture-1}

\section{Classical Landstad Duality}
\label{lecture 1}

\subsection{Motivation, Overview}

The first crossed-product duality theorem for $C^*$-algebras, due to Takai, says that if $\alpha:G\to \aut A$ is an action of an abelian locally compact group $G$ on a $C^*$-algebra $A$, then there is 
a \emph{dual action}
action $\what\alpha$ of the Pontryagin dual group $\what G$ on the crossed product $A\times_\alpha G$ such that the double crossed product is stably isomorphic to $A$:
\[
(A\times_\alpha G)\times_{\what\alpha}\what G\cong A\otimes \KK,
\]
where $\KK$ denotes the compact operators on $L^2(G)$.

Thus, having the dual action $\what\alpha$ of~$G$ 
almost allows us to recover the original $C^*$-algebra $A$
from $A\times_\alpha G$.
Not up to isomorphism, but it comes close: 
$A\otimes\KK$ has many of the same properties as $A$, 
and in particular is Morita equivalent to $A$ 
(so has the same primitive ideal space, representation theory, and $K$-theory, among other things).

Operator algebraists thought this crossed-product duality was such a good thing that they immediately began searching for a generalization to nonabelian groups, and thus the 
$C^*$-algebraic theory of \emph{coactions}
and their crossed products
was born: Imai and Takai proved that, roughly speaking, for an action $\alpha$ of \emph{any} locally compact group $G$ on $A$ there is a coaction\footnote{more about these later} $\what\alpha$ of $G$ on $A\times_\alpha G$ such that again
\[
(A\times_\alpha G)\times_{\what\alpha} G\cong  A\otimes \KK.
\]

Our main focus in these talks will be another type of duality for crossed products, proved by Landstad, that allows us to recover $A$ from the crossed product up to isomorphism, not just up to Morita equivalence. Landstad refers to it as a ``duality theorem for $C^*$-crossed products'', and nowadays the rest of us call it ``Landstad duality''.

In this first talk we'll give enough background to understand this classical Landstad duality. 
The later talks will develop a categorical framework for understanding
and extending Landstad duality.  
We expect that the philosophy and techniques we present here
in this particular context
will also be useful in other aspects
of the study of $C^*$-dynamical systems.

\subsection{Goal of Classical Landstad Duality}

The two basic questions underlying Landstad duality, as originally posed (and answered) in \cite{lan:dual}, are as follows:

\begin{q}\label{recover}
If $A\times_{\alpha,r} G$ is a reduced crossed product
by an action $\alpha$ of a given locally compact group $G$,
how do we recover $A$ and $\alpha$?
\end{q}

\begin{q}\label{identify}
When is a $C^*$-algebra $B$ the reduced crossed product of some $C^*$-algebra by some action of a given locally compact group~$G$?
\end{q}

Of course, the most meaningful interpretation of ``is'' in both questions is ``up to isomorphism''.
Moreover, the best answer to Question~\ref{recover} should involve only the 
criteria used to answer Question~\ref{identify}, not any intrinsic information 
about $A$ or $\alpha$. 

In order to see how to answer these questions, we need to look more closely at crossed products.

\subsection{Crossed products and reduced crossed products by actions}

For a more thorough exposition of the theory of actions, coactions, and their crossed products, we recommend \cite[Appendix~A]{BE}. Here we'll just give a ``hand-wavy'' introduction. Given an action $(A,G,\alpha)$, there are many constructions of the 
full crossed product $A\times_\alpha G$, but the main thing is that it has a universal property:\footnote{which is in fact the \emph{raison d'\'etre} of the full crossed product in the first place} it has the same representation theory as the covariant representations of the action.
A \emph{covariant representation} of an action $(A,\alpha)$ (note that we've dropped the notation $G$ --- the group $G$ will always be around, even if not explicitly named) is a pair $(\pi,u)$,
where $\pi$ and $u$ are representations of $A$ and $G$, respectively, on a Hilbert space, such that $\pi\circ\alpha_s=\ad u_s\circ\pi$ for $s\in G$.
More generally, instead of representations on Hilbert space, it's customary to use \emph{covariant homomorphisms}, which take values in 
the multiplier algebra $M(B)$ of a $C^*$-algebra $B$
rather than the bounded operators on Hilbert space.
The crossed product comes with a \emph{universal covariant homomorphism} $(i_A,i_G)$ taking values in $M(A\times_\alpha G)$
and generating the crossed product in the sense that
$A\times_\alpha G = \clspn\{ i_A(a) i_G(g) : a\in A, g\in C^*(G) \}$. 

%<---

We should immediately explain the \emph{universal property} of crossed products, since it's fundamental to everything we'll discuss: 
for each covariant homomorphism $(\pi,u)$ of an action $(A,\alpha)$ in $M(B)$, there is a unique 
\sk homomorphism
\[
\pi\times u:A\times_\alpha G\to M(B),
\]
called the \emph{integrated form of $(\pi,u)$},
making the diagram
\[
\xymatrix{
A \ar[r]^-{i_A} \ar[dr]_\pi
&M(A\times_\alpha G) \ar@{-->}[d]_(.4){\pi\times u}^(.4){!}
&G \ar[l]_-{i_G} \ar[dl]^u
\\
&M(B)
}
\]
commute.

%<---

Ok, that's great, but often the crossed product is too big --- the trade-off for having a nice universal property is that it's hard to ``see'' the elements of $A\times_\alpha G$. One particularly useful representation of the crossed product is the \emph{regular representation}, analogous to (and generalizing) the regular representation of $G$. The image of $A\times_\alpha G$ under the regular representation is the \emph{reduced crossed product} $A\times_{\alpha,r} G$.
The reduced crossed product doesn't capture all the representation theory of the action
(although it does encapsulate a lot of it, and in fact all of it when the group $G$ is amenable\footnote{and we won't go into the precise definition of amenability here, except to say that it's equivalent to $C^*(G)=C^*_r(G)$, and is a good thing, enjoyed by, for example, abelian and compact groups, but not by the free group on two generators --- which leads to the Banach-Tarski paradox, but that's another story, as Kipling would say})
but the elements of $A\times_{\alpha,r} G$ are ``easier to see'' because they're (often quite familiar) operators on Hilbert space.
The \emph{canonical covariant homomorphism} associated to the regular representation is denoted $(i_A^r,i_G^r)$.

By the universal property of $A\times_\alpha G$,
the reduced crossed product encodes exactly those covariant representations 
of $(A,\alpha)$
which factor through the regular representation. 
The surjective homomorphism\footnote{and the reason for the superscript ``$n$'' will be given later} of the full crossed product onto the reduced is denoted
\[
q^n:A\times_\alpha G\to A\times_{\alpha,r} G,
\]
and in fact another definition of the reduced crossed product, more useful for our purposes, is as the quotient of $A\times_\alpha G$ by the kernel of the regular representation.
Then the above surjection $q^n$ is just the quotient map.

\begin{ex}\label{cstarG}
One can view the full group $C^*$-algebra $C^*(G)$ as a crossed product
$\C\times_{\text{tr}} G$, where $\text{tr}$ is the trivial action.  
The universal property of the crossed product reduces to 
the universal property of $C^*(G)$ for (strictly continuous) unitary representations of $G$.  Under this identification, 
the reduced crossed product $\C\times_{\text{tr},r}G$ is the 
reduced group $C^*$-algebra $C^*_r(G)$.  

Throughout these notes, we will suppress the canonical map of $G$ into $M(C^*(G))$,
and simply write $s\in M(C^*(G))$ for $s\in G$. 
\end{ex}

If $\phi:A\to M(B)$ is an $\alpha-\beta$ equivariant homomorphism
(meaning that $\phi\circ\alpha_s = \beta_s\circ\phi$ for $s\in G$),
then $(i_B\circ\phi, i_G^\beta)$ is covariant for $(A,\alpha)$,
so there is a homomorphism
\[
\phi\times G=(i_B\circ\phi)\times i_G^\beta:A\times_\alpha G
\to M(B\times_\beta G).
\]
This phenomenon persists for the reduced crossed product
(although this is not obvious): 
there is a unique homomorphism $\phi\times_r G$ making the diagram
\[
\xymatrix@C+30pt{
A\times_\alpha G \ar[r]^-{\phi\times G} \ar[d]_{q^n_A}
&M(B\times_\beta G) \ar[d]^{q^n_B}
\\
A\times_{\alpha,r} G \ar@{-->}[r]_-{\phi\times_r G}^-{!}
&M(B\times_{\beta,r} G).
}
\]
commute, i.e., the kernel of $q^n_B\circ(\phi\times G)$ contains the kernel of $q^n_A$.

The crossed product carries a \emph{dual coaction} $\what\alpha$ of $G$. Being a coaction\footnote{precise definition coming later}, $\what\alpha$ is a certain type of homomorphism
\[
\what\alpha:A\times_\alpha G\to M\bigl((A\times_\alpha G)\otimes C^*(G)\bigr).
\]
The reduced crossed product carries its own version\footnote{and again the reason for the superscript ``$n$'' will come later} $\what\alpha^n$ of the dual coaction, making the diagram
\[
\xymatrix{
A\times_\alpha G \ar[r]^-{\what\alpha} \ar[d]_{q^n}
&M((A\times_\alpha G)\otimes C^*(G)) \ar[d]^{q^n\otimes\id}
\\
A\times_{\alpha,r} G \ar[r]_-{\what\alpha^n}
&M((A\times_{\alpha,r} G)\otimes C^*(G))
}
\]
commute\footnote{and to express this more compactly we say $q^n$ is 
``$\what\alpha-\what\alpha^n$
equivariant''}.

So, the reduced crossed product $A\times_{\alpha,r} G$ is a $C^*$-algebra which comes equipped with certain other gadgets: a covariant homomorphism $(i_A^r,i_G^r)$ 
of $(A,G)$ in $M(A\times_{\alpha,r}G)$, and a coaction $\what\alpha^n$ of $G$.

\subsection{Coactions, dual coactions, normal coactions}

The theory of coactions is parallel to that for actions: the \emph{crossed product} of a coaction $(A,\delta)$ is a $C^*$-algebra $A\times_\delta G$ whose representations are the same as the covariant representations of $(A,\delta)$, and more generally for the covariant homomorphisms.
To explain further we'll resort to a little bit of hand-waving (and again, for more
complete details, we recommend \cite[Appendix~A]{BE}):

Coactions of $G$ are meant to play the role for general locally compact groups
that actions of the dual group $\what G$ play when $G$ is abelian.  
To see how this works, suppose $G$ is abelian. 
Then an action $(A,\alpha)$ of $\what G$ corresponds to a certain type of \emph{comodule} structure
\[
\wilde\alpha:A\to M(A\otimes C_0(\what G)),
\]
which, after $A\otimes C_0(\what G)$ has been identified with $C_0(\what G,A)$ and the multiplier algebra with $C_b(\what G,M^\beta(A))$ (the superscript ``$\beta$'' means
$M(A)$ has the strict topology),
is given by
\[
\wilde\alpha(a)(\chi)=\alpha_{\chi}(a)\midtext{for}a\in A, \chi\in \what G.
\]
Using the (inverse) Fourier transform to take $C_0(\what G)$ isomorphically onto $C^*(G)$,
we see that an action of $\what G$ corresponds to a certain comodule structure
\[
A\to M(A\otimes C^*(G)),
\]
and abstracting this structure so that all mention of $\what G$ has been expunged gives rise to the definition of a coaction for general $G$:

\begin{defn}
\label{coaction}
A \emph{coaction} of $G$ on a $C^*$-algebra $A$ is an
injective nondegenerate homomorphism
\[
\delta:A\to M(A\otimes C^*(G))
\]
satisfying
\begin{enumerate}
\item $\clspn\{\delta(A)(1\otimes C^*(G))\}=A\otimes C^*(G)$;

\item $(\delta\otimes\id)\circ\delta=(\id\otimes\delta_G)\circ\delta$,
where $\delta_G:C^*(G)\to M(C^*(G)\otimes C^*(G))$ is the homomorphism corresponding to the 
strictly continuous unitary homomorphism of $G$ into $M(C^*(G)\otimes C^*(G))$
given by%
\footnote{Recall that we identify $s\in G$ with its canonical image in $M(C^*(G))$.}
\[
s \mapsto s\otimes s.
\]
\end{enumerate}
\end{defn}

Item~(1) is a technical condition that makes crossed-product duality work\footnote{and used to be called \emph{nondegeneracy} of the coaction}, and (2) is a kind of ``co-associativity'' corresponding to the property that, when $G$ is abelian, an action of $\what G$ on $A$ is a \emph{homomorphism} from $\what G$ to the automorphism group $\aut A$.
It's important to observe that $\delta_G$ itself is a coaction of $G$ on $C^*(G)$, called the \emph{canonical coaction} (or \emph{comultiplication}) on $C^*(G)$.

\begin{ex}
The \emph{dual coaction} $\what\alpha^n$ on the reduced crossed product $A\times_{\alpha,r} G$ of an action $(A,\alpha)$ is the homomorphism
\[
\what\alpha^n=(i_A^r\otimes 1)\times (i_G^r\otimes\id):
A\times_{\alpha,r} G\to
M\bigl((A\times_{\alpha,r} G)\otimes C^*(G)\bigr).
\]
Thus on the generators we have
\begin{align}
\what\alpha^n(i_A^r(a))&=i_A^r(a)\otimes 1\midtext{for}a\in A
\label{alpha hat A}\\
\what\alpha^n(i_G^r(s))&=i_G^r(s)\otimes s\midtext{for}s\in G.
\label{alpha hat G}.
\end{align}
In the next lecture it'll be important to note that
\eqref{alpha hat G} implies that the homomorphism $i_G^r:C^*(G)\to M(A\times_{\alpha,r} G)$ is $\delta_G-\what\alpha^n$ equivariant.
\end{ex}

It's not so easy to describe the covariant homomorphisms of a coaction, and again we'll resort to hand-waving.
As a warm-up, let's rephrase covariance for actions: 
a covariant homomorphism $(\pi,u)$ in a multiplier algebra $M(B)$
of an action $(A,G,\alpha)$ can be regarded as 
a homomorphism  $\pi:A\to M(B)$
that is $\alpha-\ad u$ equivariant:
\[
 \pi(\alpha_s(a)) = u_s \pi(a) u_s^* = (\ad u)_s(\pi(a)),
\]
where $\ad u:G\to \aut B$ is the \emph{inner action} implemented by the unitary homomorphism $u:G\to M(B)$.

If $G$ is abelian and $\alpha$ is an action of $\what G$ on $A$, 
then a covariant homomorphism involves a homomorphism $u:\what G\to M(B)$,
and by the universal property of the group $C^*$-algebra, 
this corresponds to a homomorphism\footnote{and here 
we've made a choice to abuse notation by using $u$ for two different maps; 
it's a fairly common trade-off.} $u:C^*(\what G)\to M(B)$.
Fourier transforming, we get a homomorphism $\mu:C_0(G)\to M(B)$, 
which implements (what must be)
an \emph{inner coaction} $\ad\mu$ of $G$ on $B$. 
Generalizing to (possibly) nonabelian groups~$G$, 
a \emph{covariant homomorphism} in a multiplier algebra $M(B)$
of a coaction $(A,G,\delta)$ 
can be defined as a pair $(\pi,\mu)$, where
$\mu:C_0(G)\to M(B)$ is a homomorphism and
$\pi:A\to M(B)$ is 
$\delta-\ad\mu$ equivariant:
\[
(\pi\otimes\id)(\delta(a))  = (\ad\mu)(\pi(a)).
\]
There is a \emph{universal covariant homomorphism} $(j_A,j_G)$ of the coaction $(A,\delta)$ in $M(A\times_\delta G)$.

The universal property for crossed products of coactions is parallel to that for actions: given a covariant homomorphism $(\pi,\mu)$ of a coaction $(A,\delta)$ in $M(B)$, there is a unique \emph{integrated form}
\[
\pi\times\mu:A\times_\delta G\to M(B)
\]
making the diagram
\[
\xymatrix{
A \ar[r]^-{j_A} \ar[dr]_\pi
&M(A\times_\delta G) \ar@{-->}[d]_(.4){\pi\times\mu}^(.4){!}
&C_0(G) \ar[l]_-{j_G} \ar[dl]^\mu
\\
&M(B)
}
\]
commute.

It turns out that coactions come in various flavors%
\footnote{although crossed products by coactions do not!}%
, and here we want the \emph{normal} ones, which means that the canonical homomorphism $j_A:A\to M(A\times_\delta G)$ is injective; this not only makes life significantly easier, but is necessary for the (imminent) statement of the ``original'' Landstad duality.

\begin{ex}\label{normal}
The dual coaction $\what\alpha^n$ of $G$ on a \emph{reduced} crossed product $A\times_{\alpha,r}G$ by an action of $G$ is always normal;
in particular, 
the canonical coaction $\delta_G^n$ on $C_r^*(G) = \C\times_{\text{tr},r}G$
is always normal.
In general, if $A\times_\alpha G \neq A\times_{\alpha,r}G$,
then the dual coaction $\what\alpha$ on the full crossed product is not normal.

For any $A$ and $G$, the trivial coaction $a \mapsto a\otimes 1$ of $G$ on $A$
is normal, because the canonical map $j_A$ becomes the canonical embedding
$a\mapsto a\otimes 1$ of $A$ into $M(A\otimes C_0(G))$.
\end{ex}

%<---

% \subsection{Crossed-product duality}
% 
% 
% the actual definition of the dual coaction $\what\alpha^n$: it is the morphism from $A\times_{\alpha,r} G$ to $(A\times_{\alpha,r} G)\otimes C^*(G)$ in \cs\ determined by the covariant homomorphism
% \[
% \bigl(i_A^r\otimes 1,i_G^r\otimes\id\bigr)
% \]
% of $(A,\alpha)$ in $M((A\times_{\alpha,r} G)\otimes C^*(G))$, i.e.,
% \begin{align}
% \what\alpha^n(i_A^r(a))&=i_A^r(a)\otimes 1\midtext{for}a\in A
% \label{alpha hat A}\\
% \what\alpha^n(i_G^r(s))&=i_G^r(s)\otimes s\midtext{for}s\in G.
% \label{alpha hat G}.
% \end{align}
% The important observation here is that
% \eqref{alpha hat G} implies that the morphism $i_G^r:C^*(G)\to A\times_{\alpha,r} G$ is $\delta_G-\what\alpha^n$ equivariant, 

%<---

\subsection{Classical Landstad Duality}

We have seen that every reduced crossed product $A\times_{\alpha,r} G$
comes with a certain unitary homomorphism
$i_G^r$ of $G$ into $M(A\times_{\alpha,r} G)$
and a dual coaction $\what\alpha^n$ of $G$
which are related to one another by~\eqref{alpha hat G}.
Landstad's answer to Question~\ref{identify}
is that having this extra structure
entirely characterizes the reduced crossed products by~$G$. 
(Landstad's answer to Question~\ref{recover}
was a technical \emph{tour de force} by which an action isomorphic to $(A,\alpha)$
was explicitly constructed from $A\times_{\alpha,r}G$ using $i_G^r$ and $\what\alpha^n$.)

This theorem, originally proven by Landstad in \cite[Theorem~3]{lan:dual}, has become one of the standard results in $C^*$-dynamical systems, especially regarding ``noncommutative duality''.
The version we state here appeared 
in \cite{kqfulllandstad} as a stepping stone to Landstad duality for \emph{full} crossed products\footnote{but that is another story, as Kipling would say, which we will tell in a later talk.};
it's a significant restructuring of Landstad's version, 
and we needed to work a bit to prove it. 

\begin{thm}
[{\cite[Theorem~3.1]{kqfulllandstad}}]
\label{landstad normal}
Let $B$ be a $C^*$-algebra and let $G$ be a locally compact group.
Then
\begin{enumerate}
\item
There exist an action $(A,\alpha)$ of $G$ 
and an isomorphism $\theta\colon B\to A\times_{\alpha,r}G$
if and only if 
there 
are
a strictly continuous unitary homomorphism $u:G\to M(B)$ 
and a normal coaction $\delta$ of $G$ on $B$ such that
\begin{align*}
\delta(u_s)&=u_s\otimes s
\quad\text{for all $s\in G$.}
%\label{delta u}
\end{align*}
Moreover, 
we can choose $A$, $\alpha$, and $\theta$ such that
$\theta$ is $\delta-\what\alpha^n$ equivariant and $\theta\circ u=i_G^{\alpha,r}$.

\item
If $(C,\beta)$ is another action
and $\sigma:B\to C\times_{\beta,r} G$ is a $\delta-\what\beta^n$ equivariant isomorphism 
such that $\sigma\circ u=i_G^{\beta,r}$, 
then there is a $\alpha-\beta$ equivariant isomorphism $\phi:A\to C$ 
such that $(\phi\times_r G)\circ\theta=\sigma$.
\end{enumerate}
\end{thm}

Ok, that was a mouthful;
the ``moreover'' and 
the uniqueness clause (2) are particularly cumbersome.
In the next talk we'll restate Landstad's theorem in categorical language, and one benefit will be a much more economical formulation of the uniqueness.

%----------------------------------------------------------------

\newpage

\section{Categorical Landstad duality for reduced crossed products}
\label{lecture 2}

Here is the categorical version of Landstad's theorem:

\begin{thm}
[{\cite[Theorem~4.1]{clda}}]
\label{landstad category}
The functor
\begin{align*}
(A,\alpha)&\mapsto (A\times_{\alpha,r} G,\what\alpha^n,i_G^r)\\
\phi&\mapsto \phi\times_r G
\end{align*}
is an equivalence from the category 
of actions \(of $G$\)
to the comma category 
of normal coactions \(of $G$\)
under the canonical coaction $\delta_G$ on $C^*(G)$.
\end{thm}

\begin{qs}\

\begin{enumerate}
\item What does that all mean?

\item Why is categorical language appropriate here?
\end{enumerate}
\end{qs}

We'll answer the above two questions out of order.
For the second question, 
% we analyze what Landstad's theorem does.
recall that Landstad duality tells us that we can recover
\sknote{Again with the ``recover''.  I guess I feel ``recovery'' should
involve an explicit construction, not just an abstract proof of existence.}
$(A,\alpha)$ up to isomorphism from the data $(A\times_{\alpha,r} G,\what\alpha^n,i_G^r)$.
Now, it transpires that often it's quite cumbersome to make precise what ``up to isomorphism'' means (e.g., have a look at the uniqueness clause in Landstad's theorem).
Here is where category theory can help: it turns out that there's a category of actions and a category of normal coactions, and we can augment the latter category so that the above data gives an equivalence between these categories, and moreover from this we can recover Landstad duality, plus more.

But what does it mean to say that the crossed-product functor
$\cp$ from actions to the comma category is an equivalence?
It would be an \emph{isomorphism} of categories if there were a functor $F$ in the opposite direction\footnote{and we won't need to contemplate what $F$ might be}
such that
each composition $\cp\circ F$ and $F\circ\cp$ is the identity functor.
% But what is an ``equivalence of categories''?
% Abstractly, categories $\CC$ and $\DD$
% are \emph{isomorphic}
% if there are functors
% \[
% F:\CC\to \DD\midtext{and}G:\DD\to \CC
% \]
% such that $F\circ G=\id_\CC$ and $G\circ F=\id_\DD$.
But it turns out that this notion is not very useful\footnote{in the sense that, although some categories arising in nature are actually isomorphic, this doesn't happen often enough to be very interesting}; it's much better to ask,
for each action $(A,\alpha)$,
not that 
$F\circ\cp(A,\alpha)$ 
and 
$(A,\alpha)$
% for each object $x$ in $\CC$, 
% not that $GF(x)$ and $x$ 
coincide, but rather that they be isomorphic, 
and similarly 
% to ask 
for the composition $\cp\circ F$.
% $FG(y)\cong y$ for each object $y$ in $\DD$.
However\footnote{and here's the categorical perspective}, 
this is not quite enough: 
the crucial thing is that we want these isomorphisms to be ``natural''.
This means that we want isomorphisms $\theta$ making the diagrams
\[
\xymatrix@C+30pt{
(A,\alpha)
\ar[r]^\phi \ar[d]_\theta^\cong
&(B,\beta) \ar[d]^\theta_\cong
\\
F\circ\cp(A,\alpha)
\ar[r]_{F\circ\cp(\phi)}
&F\circ\cp(B,\beta)
}
\]
% \[
% \xymatrix{
% GFx \ar[r]^-{Ff} \ar[d]_{\theta_x}^\cong
% &GFy \ar[d]^{\theta_y}_\cong
% \\
% x \ar[r]_f
% &y
% }
% \]
commute, for every morphism
$\phi$ of actions.
% $f:x\to y$ in $\CC$.
Then we say $F\circ\cp$ is \emph{naturally isomorphic} to the identity functor.
And we similarly want $\cp\circ F\cong\id$.
When such a functor $F$ exists, we say that $\cp$ is an \emph{equivalence} of categories. 

What are the categories of interest to us here?
We start with

\begin{defn}
The \emph{basic $C^*$-category} \cs\ has:
\begin{enumerate}
\item objects: $C^*$-algebras;

\item morphisms: nondegenerate homomorphisms into multiplier algebras.
\end{enumerate}
\end{defn}

In more detail: if $A$ and $B$ are objects in \cs\ (i.e., they are $C^*$-algebras), then a \emph{morphism} $\phi:A\to B$ in \cs\ is a nondegenerate homomorphism $\phi:A\to M(B)$.
Here \emph{nondegenerate} means $\phi(A)B=B$.
It turns out that the naive idea of using ordinary homomorphisms between the $C^*$-algebras themselves is too restrictive for many purposes.
Of course, one should check that \cs\ really is a category, i.e., that it has identity morphisms (obvious) and that morphisms can be composed (not so obvious, but not deep).

A fundamental property of the basic category is that its isomorphisms are familiar:

\begin{lem}
A morphism $\phi:A\to B$ in the category \cs\ is an isomorphism if and only if $\phi$ maps $A$ into $B$ and is an isomorphism of $C^*$-algebras in the usual sense.
\end{lem}

We need equivariant versions of the basic category; here is the appropriate version for actions:

\begin{defn}
The category \act\ has:
\begin{enumerate}
\item objects: actions of $G$;

\item morphisms: equivariant morphisms in \cs.
\end{enumerate}
\end{defn}

In more detail: if $(A,\alpha)$ and $(B,\beta)$ are actions of $G$, then a \emph{morphism} $\phi:(A,\alpha)\to (B,\beta)$ in \act\ is a morphism $\phi:A\to B$ in \cs\ that is \emph{$\alpha-\beta$ equivariant} in the sense that
\[
\phi\circ\alpha_s=\beta_s\circ\phi\midtext{for all}s\in G.
\]

Dually, the appropriate equivariant category for coactions is:

\begin{defn}
The category \coact\ has:
\begin{enumerate}
\item objects: coactions of $G$;

\item morphisms: equivariant morphisms in \cs.
\end{enumerate}
\end{defn}

In more detail: if $(A,\delta)$ and $(B,\epsilon)$ are coactions of $G$, then a \emph{morphism} $\phi:(A,\delta)\to (B,\epsilon)$ in \coact\ is a morphism $\phi:A\to B$ in \cs\ that is \emph{$\delta-\epsilon$ equivariant} in the sense that
the diagram
\[
\xymatrix{
A \ar[r]^-\delta \ar[d]_\phi
&A\otimes C^*(G) \ar[d]^{\phi\otimes\id}
\\
B \ar[r]_-\epsilon
&B\otimes C^*(G)
}
\]
commutes in \cs,
which makes sense because
\[
\phi\otimes\id:A\otimes C^*(G)\to B\otimes C^*(G)
\]
is a morphism in \cs.

As mentioned in the first lecture, some coactions are \emph{normal}, in the sense that the canonical morphism $j_A:A\to A\times_\delta G$ in \cs\ is a monomorphism, and these form a category:

\begin{defn}
The category \ncoact\ is the full subcategory of \coact\ whose objects are the normal coactions of $G$.
\end{defn}

What are the ``crossed-product functors''?

\begin{defn}
The functor
\[
\cp:\act\to\coact
\]
takes each action $(A,\alpha)$ to the dual coaction $(A\times_\alpha G,\what\alpha)$,
and each morphism $\phi:(A,\alpha)\to (B,\beta)$ in \act\ to the morphism
\[
\phi\times G:=(i_B\circ\phi)\times i_G:
(A\times_\alpha G,\what\alpha)\to (B\times_\beta G,\what\beta)
\]
in \coact.
\end{defn}

For Landstad's theorem, we need to use reduced crossed products:

\begin{defn}
The functor
\[
\rcp:\act\to\ncoact
\]
takes each action $(A,\alpha)$ to the dual coaction $(A\times_{\alpha,r} G,\what\alpha^n)$,
and each morphism $\phi:(A,\alpha)\to (B,\beta)$ in \act\ to the morphism
\[
\phi\times_r G:=(i_B^r\circ\phi)\times i_G^r:
(A\times_{\alpha,r} G,\what\alpha^n)\to (B\times_{\beta,r} G,\what\beta^n)
\]
in \ncoact.
\end{defn}

But in \thmref{landstad category} we need this functor to take values in an ``enhanced'' category, not just \ncoact\ ---
namely, we need our functor to keep track of $i_G^r$.
Well, the unitary homomorphism $i_G^r:G\to M(A\times_{\alpha,r} G)$ corresponds to a morphism
\[
i_G^r:C^*(G)\to A\times_{\alpha,r} G\midtext{in}\cs,
\]
which turns out to have a special property:
it follows from the definition of the dual coaction
--- see \eqref{alpha hat G} ---
that $i_G^r$ is $\delta_G-\what\alpha^n$ equivariant, 
i.e., gives a morphism in \coact: 
\[
i_G^r:(C^*(G),\delta)\to (A\times_{\alpha,r} G,\what\alpha^n).
\]

% But there's a subtlety: $i_G^r$ is \emph{not} a morphism in the subcategory \ncoact, for the simple reason that the canonical coaction $\delta_G$ is typically not normal. In fact, $\delta_G$ is normal exactly when the group $G$ is amenable. More generally, the dual coaction $(A\times_\alpha G,\what\alpha)$ is normal if and only if $A\times_\alpha G=A\times_{\alpha,r} G$; this happens if $G$ is amenable, and in some other special cases, but not always.

Anyway, we now have a category, namely \coact, and a specific object in that category, namely $(C^*(G),\delta_G)$, and we are considering morphisms, of the form $i_G^r$, from that object into a particular subcategory, namely \ncoact.
This turns out to be an example of something category theorists call a \emph{comma category}.

In our case the comma category is
\[
(C^*(G),\delta_G)\dn\ncoact,
\]
with
\begin{enumerate}
\item objects: morphisms in \coact\ from $(C^*(G),\delta_G)$ to objects in the subcategory \ncoact;

\item morphisms: morphisms $\sigma$ in \coact\ making a diagram of the form
\[
\xymatrix{
&(C^*(G),\delta_G) 
\ar[dl]_\phi 
\ar[dr]^\psi
\\
(A,\delta) 
\ar[rr]_\sigma
&&(B,\epsilon)
}
\]
commute.
\end{enumerate}

We'll write an object in the comma category $(C^*(G),\delta_G)\dn\ncoact$  as a triple $(A,\delta,\phi)$, where $\phi:(C^*(G),\delta_G)\to (A,\delta)$.
% An object in the comma category $(C^*(G),\delta_G)\dn\ncoact$ is officially a pair
% $((A,\delta),\phi)$,
% but we'll simplify the notation by writing the object as a triple $(A,\delta,\phi)$ instead of a pair.

Thus, we now see that, given an action $(A,\alpha)$, the triple $(A\times_{\alpha,r} G,\what\alpha^n,i_G^r)$ is an object in $(C^*(G),\delta_G)\dn\ncoact$, and so finally we can understand the functor in \thmref{landstad category}, and the assertion is that it is a category equivalence.
Here is a re-statement of Theorem~\ref{landstad category} in the notation we've introduced:

\begin{thm}
The functor
\begin{align*}
\rcp:(A,\alpha)&\mapsto (A\times_{\alpha,r} G,\what\alpha^n,i_G^r)\\
\phi&\mapsto \phi\times_r G
\end{align*}
is an equivalence from \act\ to the comma category $(C^*(G),\linebreak[0]\delta_G)\dn \ncoact$.
\end{thm}

\newpage

\section{Proofs of CLDA, and CLDC}
\label{lecture 3}

\subsection{CLDA}

At the end of the last talk, we saw the statement of CLDA
(Categorical Landstad Duality for Actions).
Now we'll indicate the idea of the proof,
after which we'll see the ``dual'' version CLDC
(Categorical Landstad Duality for Coactions).

\begin{thm}
\label{actionnormal}
The functor
\begin{align*}
\rcp:(A,\alpha)&\mapsto (A\times_{\alpha,r} G,\what\alpha^n,i_G^r)\\
\phi&\mapsto \phi\times_r G
\end{align*}
is an equivalence from \act\ to the comma category $(C^*(G),\linebreak[0]\delta_G)\dn \ncoact$.
\end{thm}

\begin{proof}
[Outline of proof]
We have to begin by choosing a strategy for proving that a functor is an equivalence. The definition would require us to find a \emph{quasi-inverse} functor going the other way, and verify that both compositions are naturally isomorphic to the identity functor.
This adds up to a lot of work, and fortunately we can choose among various short-cuts.
This time we apply a standard result from category theory (and for all this we recommend \cite{maclane}): it suffices to show that $\rcp$ is \emph{full}, \emph{faithful}, and \emph{essentially surjective}.

Essential surjectivity means that every object in $(C^*(G),\d_G)\dn \ncoact$ is isomorphic to one in the image of $\rcp$,
and this is part of the content of classical Landstad duality, which we stated at the end of the first lecture.

Fullness means that for all actions $(A,\alpha)$ and $(B,\beta)$, 
the functor $\rcp$ maps the set of morphisms $\mor((A,\alpha),(B,\beta))$ 
\emph{onto} the set of morphisms $\mor(\rcp(A,\alpha),\rcp(B,\beta))$, 
and faithfulness means that the same map is injective. 
Together, fullness and faithfulness provide
a generalization of the uniqueness clause in classical Landstad duality.

Thus, we need to show that if $(A,\alpha)$ and $(B,\beta)$ are actions and
\[\sigma:(A\times_{\alpha,r} G,\what\alpha^n,i^{\alpha,r}_G)\to
(B\times_{\beta,r} G,\what\beta^n,i^{\beta,r}_G)\]
is a morphism in $(C^*(G),\d_G)\dn \ncoact$, then there exists a
unique morphism
$\p:(A,\alpha)\to (B,\beta)$
in $\act$ such that $\s=\p\times_r G$.
Notice that if this were the case, we would have 
\[
\sigma\circ i_A^r = (\phi\times_r G)\circ i_A^r
= ((i_B^r\circ\phi)\times i_G^{\beta,r})\circ i_A^r
= i_B^r\circ\phi,
\]
so consider the diagram
\[\xymatrix{
A \ar[r]^-{i^r_A} \ar@{-->}[d]^{!}_\p
&M(A\times_{\alpha,r} G) \ar[d]^{\s}
\\
M(B) \ar[r]_-{i^r_B}
&M(B\times_{\beta,r} G).
}\]
Our strategy in \cite{clda} comprised the following steps:
\begin{enumerate}
\item Show that there is a unique homomorphism $\p$ making the diagram commute.
For this we observed that, since $i_B^r$ is faithful, it sufficed to show that
\[
\sigma\circ i_A^r(A)\subset i_B^r(M(B)),
\]
and for this we made heavy use of Landstad's techniques from \cite{lan:dual}.
The fact that $\sigma$ respects the dual coactions and unitary homomorphisms 
--- i.e., that $\sigma$ is a morphism in the comma category --- is critical here:
in general, not every nondegenerate homomorphism 
of $A\times_{\alpha,r}G$ into $B\times_{\beta,r}G$ is going to
be of the form $\phi\times_r G$.

\item Show that $\phi$ is $\alpha-\beta$ equivariant and nondegenerate
(here again it is critical that $\sigma$ is a morphism in the comma category);

\item Observe that by construction we have $\phi\times_r G=\sigma$.
\end{enumerate}
\end{proof}

\subsection{``Deducing'' classical Landstad duality from CLDA}
Of course, Landstad's original theorem was used heavily in our proof of CLDA,
but it is instructive to see how the classical version can be recovered
from the categorical version.

First suppose we have an isomorphism $\theta:B\to A\times_{\alpha,r} G$. Then the diagrams
\[
\xymatrix{
G \ar@{-->}[r]^-u \ar[dr]_(.4){i_G^r}
&M(B) \ar[d]^{\bar\theta}_\cong
\\
&M(A\times_{\alpha,r} G)
}
\]
and
\[
\xymatrix{
B \ar@{-->}[r]^-\delta \ar[d]_\theta^\cong
&M(B\otimes C^*(G)) \ar[d]^{\bar{\theta\otimes\id}}_\cong
\\
A\times_{\alpha,r} G \ar[r]_-{\what\alpha^n}
&M((A\times_{\alpha,r} G)\otimes C^*(G)),
}
\]
where $\bar\theta:M(B)\to M(A\times_{\alpha,r} G)$ is the canonical extension of $\theta$ to multipliers\footnote{whose existence is vouchsafed by nondegeneracy}, and similarly for $\bar{\theta\otimes\id}$,
can obviously be completed, giving the maps $u$ and $\delta$ with the required properties (and the coaction $\delta$ is normal because it's isomorphic to the normal coaction $\what\alpha^n$).

Conversely, given the existence of $u$ and $\delta$, 
by essential surjectivity of $\rcp$ we have 
$(A,\alpha)\in \AA(G)$
and an isomorphism
\[
\theta:(B,\delta,u)\to (A\times_{\alpha,r} G,\what\alpha^n,i_G^r)\midtext{in}(C^*(G),\delta_G)\dn\ncoact,
\]
i.e., the isomorphism $\theta:B\to A\times_{\alpha,r} G$ is $\delta-\what\alpha^n$ equivariant and satisfies $\theta\circ u=i_G^r$. Moreover, if we also have $(C,\beta)$ and $\sigma$, then the commutative (by construction!) diagram
\[
\xymatrix@C+30pt{
(B,\delta,u)
\ar[r]^-\theta_-\cong
\ar[dr]_\sigma^\cong
&(A\times_{\alpha,r} G,\what\alpha^n,i_G^r)
\ar[d]^{\sigma\circ\theta\inv}
\\
&(C\times_{\beta,r} G,\what\beta^n,i_G^r)
}
\]
in $(C^*(G),\delta_G)\dn\ncoact$ implies, by fullness, that there is a morphism $\phi:(A,\alpha)\to (C,\beta)$ in \act\ such that
\[
\sigma\circ\theta\inv=\phi\times_r G,
\]
so that
\[
(\phi\times_r G)\circ\theta=\sigma,
\]
and moreover $\phi$ is an isomorphism because $\phi\times_r G$ is, since \RCP\ is an equivalence.

%<---

\subsection{CLDC}

What about a dual version of Landstad duality, starting with coactions?
First, 
just as every crossed product by an action of an abelian group $G$ 
carries a dual action of the Pontryagin dual group $\what G$,
every coaction crossed product $A\times_\delta G$ carries a \emph{dual action}
$\what\delta\colon G\to \aut(A\times_\delta G)$
defined by
\[
\what\delta_s=j_A\times (j_G\circ\rt_s)\quad\text{for $s\in G$,}
\]
where $\rt$ is the action of $G$ on $C_0(G)$ by right translation:
\[
\rt_s(f)(t)=f(ts).
\]
Note that by definition, $\what\delta_s\circ j_G = j_G\circ \rt_s$
for all $s\in G$,
so $j_G\colon C_0(G) \to M(A\times_\delta G)$ is 
$\rt - \what\delta$ equivariant.

The noncategorical version 
of Landstad duality for coactions
should give a characterization (up to isomorphism) of crossed products by coactions,
together with a uniqueness clause.
The following theorem originally appeared in \cite[Theorem~3.3]{Q:landstad}. We've modified the original statement considerably, using modern notation and terminology, and structuring it to be parallel to classical Landstad duality for actions.

\begin{thm}
\label{landstad coaction modern}
Let $B$ be a $C^*$-algebra and let $G$ be a locally compact group. 
Then
\begin{enumerate}
\item
There exist a normal coaction $(A,\delta)$ of $G$ 
and an isomorphism $\theta\colon B\to A\times_\delta G$
if and only if 
there 
are
an action $\alpha$ of $G$ on $B$ and a $\rt-\alpha$ equivariant morphism $\mu:C_0(G)\to B$.

Moreover, we can choose $A$, $\delta$, and $\theta$ such that
$\theta$ is $\alpha-\what\delta$ equivariant and $\theta\circ\mu=j_G^\delta$.

\item If $(C,\epsilon)$ is another normal coaction 
and $\sigma:B\to C\times_\epsilon G$ is an $\alpha-\what\epsilon$ equivariant isomorphism such that $\sigma\circ\mu=j_G^\epsilon$, 
then there is a $\delta-\epsilon$ equivariant isomorphism $\phi:A\to C$ 
such that $(\phi\times G)\circ\theta=\sigma$.
\end{enumerate}
\end{thm}

And now we'll state the categorical version (CLDC).  
First, the relevant functor:

\begin{defn}
The functor
\[
\cp:\coact\to\act
\]
takes each coaction $(A,\delta)$ to the 
dual action $(A\times_\delta G,\what\delta)$,
and each morphism $\phi:(A,\delta)\to (B,\epsilon)$ in \coact\ to the morphism
\[
\phi\times G:=(j_B\circ\phi)\times j_G:
(A\times_\delta G,\what\delta)\to (B\times_\epsilon G,\what\epsilon)
\]
in \act.
\end{defn}

\begin{thm}
[{\cite[Theorem~4.2]{cldx}}]
\label{landstad coaction category}
The functor
\begin{align*}
\cp:(A,\delta)&\mapsto (A\times_\delta G,\what\delta,j_G)\\
\phi&\mapsto \phi\times G
\end{align*}
is an equivalence from \ncoact\ to the comma category $(C_0(G),\rt)\dn\act$.
\end{thm}

Actually, \cite[Theorem~4.2]{cldx} is phrased in terms of \emph{reduced coactions}, which are just like (full) coactions but use $C^*_r(G)$ rather than $C^*(G)$. It turns out that this causes no confusion or difficulty, because there is complete freedom of choice between reduced and full coactions \cite{fullred}, so the preceding theorem follows more-or-less immediately from the one appearing in \cite{cldx}.

The proof of CLDC is \emph{not} parallel to that of CLDA: in \cite{cldx} we construct a quasi-inverse functor $\fix:(C_0(G),\rt)\dn\act\to\ncoact$.
But again, we appeal to a short-cut to avoid the drudgery of showing that \emph{both} compositions are naturally isomorphic to the identity:
standard category theory tells us that it suffices to show that $\fix$ is full and faithful, and that
$\cp\circ\fix\cong\id$,
and this is what we did.
This strategy has the advantage of displaying a quasi-inverse --- an explicit functor which 
``undoes'' the crossed-product functor up to isomorphism --- and moreover it does so without
explicitly computing that $\fix\circ\cp\cong\id$. 
We would like to do this for CLDA, although it would be much more cumbersome. 
(In CLDC, we could construct $\fix$ by carefully ``averaging'' over the dual action, 
whereas it is difficult to even formulate what averaging over a coaction would entail.)

\newpage

\section{Categorical Landstad duality for full crossed products}
\label{lecture 4}

\subsection{Crossed-product duality for coactions}

Crossed-product duality for actions was used to motivate Landstad duality for actions, 
and it will behoove us,
not just for motivation but also for future use in these lectures,
to describe the analogous crossed-product duality for coactions.
The original version is due to Katayama\footnote{fine print: Katayama used \emph{reduced} coactions, but subsequently we've worked out an equivalent version using \emph{normal} ones}.
To state it in modern terms, and again to help us in future lectures, we start with an arbitrary (i.e., not necessarily normal) coaction $(A,\delta)$.
Nilsen showed that there is a \emph{canonical surjection}
\[
\Phi:A\times_\delta G\times_{\what\delta} G\to A\otimes\KK.
\]
We can give an equivalent, albeit nonstandard, statement of Katayama duality:

\begin{thm}
[{see \cite[Proposition~2.2]{ekq}}]
The coaction $\delta$ is normal if and only if the canonical surjection $\Phi$ descends to an isomorphism
\[
A\times_\delta G\times_{\what\delta,r} G\cong A\otimes\KK.
\]
\end{thm}

Thus, paralleling the situation for actions, Katayama duality allows us to recover 
\sk $A$ 
up to Morita equivalence from the crossed product $A\times_\delta G$, whereas Landstad duality for coactions allows recovery up to isomorphism.

\subsection{Full CLDA}

Landstad duality for reduced crossed products shows how to recover an action up to isomorphism from the reduced crossed product, plus some additional information. 
What about a version for the full crossed product?
Here is the categorical version:

\begin{thm}
[{\cite[Theorem~5.1]{clda}}]
\label{full landstad category}
The functor
\begin{align*}
\cp:(A,\alpha)&\mapsto (A\times_\alpha G,\what\alpha,i_G)\\
\phi&\mapsto \phi\times G
\end{align*}
is an equivalence from the category \act\
of actions \(of $G$\)
to the comma category $(C^*(G),\delta_G)\dn\mcoact$
of maximal coactions \(of $G$\)
under the canonical coaction $\delta_G$ on $C^*(G)$.
\end{thm}

To explain this theorem, we need to introduce maximal coactions.
\sknote{I deleted a redundant discussion of $\Phi$.}
% Recall that for every coaction $(A,\delta)$ we have a canonical surjection
% \[
% \Phi:A\times_\delta G\times_{\what\delta} G\to A\otimes\KK(L^2(G)),
% \]
% and that $\delta$ is normal exactly when $\Phi$ descends to an isomorphism $\Phi^r$ making the diagram
% \[
% \xymatrix{
% A\times_\delta G\times_{\what\delta} G
% \ar[r]^-\Phi \ar[d]_\Lambda
% &A\otimes \KK(L^2(G))
% \\
% A\times_\delta G\times_{\what\delta,r} G
% \ar@{-->}[ur]_{\Phi^r}^\cong
% }
% \]
% commute,
% where $\Lambda$ is the regular representation from the full crossed product to the reduced one.
% 
% For some coactions, $\Phi$ itself is an isomorphism:

\begin{defn}
A coaction $(A,\delta)$ is \emph{maximal} if the canonical surjection $\Phi$ is an isomorphism.
\end{defn}

We'll explain more fully later, but roughly speaking the reason for the name ``maximal'' is that every coaction $(A,\delta)$ sits in a diagram of the following form:
\[
\xymatrix{
(A^m,\delta^m) \ar[r]
&(A,\delta) \ar[r]
&(A^n,\delta^n),
}
\]
where the coactions on the left and right ends are maximal and normal, respectively, the arrows are surjections, and all three crossed products are isomorphic.\footnote{Thus, normal coactions can be thought of as ``minimal''.}

\begin{ex}\label{maximal}
The dual coaction $\what\alpha$ on the full crossed product 
$A\times_\alpha G$ by an action is automatically maximal --- this was proved in \cite[Proposition~3.4]{ekq}.
In particular, the canonical coaction $(C^*(G),\delta_G)$ is maximal, since $\delta_G$ is the dual coaction on the crossed product
\[
C^*(G)=\C\times_{\text{tr}} G
\]
of the trivial action on $\C$.

For any $A$ and $G$, the trivial coaction $a\mapsto a\otimes 1$ is also maximal,
since the canonical surjection becomes
\[
 \id\otimes T\colon A\otimes C_0(G)\otimes C^*(G) \to A\otimes K(L^2(G)),
\]
where $T$ is the canonical isomorphism.
(Recall from Example~\ref{normal} that the trivial coaction is normal as well.)
\end{ex}

We gather the maximal coactions into a category:

\begin{defn}
\mcoact\ is the full subcategory of \coact\ whose objects are the maximal coactions of $G$.
\end{defn}

Thus $(C^*(G),\delta_G)$ is an object in \mcoact, and 
\thmref{full landstad category} says that
\sk
the crossed-product functor
\[
\cp:\act\to \bigl((C^*(G),\delta_G)\dn\mcoact\bigr)
\]
is an equivalence of categories.
As such, $\cp$ 
is full, faithful, and essentially surjective.
As we've seen in the earlier versions of Landstad duality, the essential surjectivity gives a characterization of full crossed products:

\begin{cor}
[{\cite[Theorem~3.2]{kqfulllandstad}}]
\label{full landstad}
Let $B$ be a $C^*$-algebra and $G$ a locally compact group.
Then
$B$ is isomorphic to a full crossed product $A\times_\alpha G$ by an action 
\sk 
if and only if there are a strictly continuous unitary homomorphism $u:G\to M(B)$ and a maximal coaction $\delta$ of $G$ on $B$ such that
\[
\delta(u_s)=u_s\otimes s.
% \label{delta u}
\]
\end{cor}

We could include a uniqueness clause in the above characterization, but it would be awkward, and is captured much more cleanly by the fullness and faithfulness properties of $\cp$.

The versions of categorical Landstad duality for full and reduced crossed products are obviously parallel, but they aren't proven using parallel techniques.
More precisely, we don't just adapt the proof of reduced Landstad duality to prove \thmref{full landstad category} ---
rather, we \emph{deduce} it from reduced Landstad duality,
with the help of another category equivalence:
\[
\mcoact\sim\ncoact,
\]
which we will discuss in the next lecture.
But to give some idea of what's going on, let's informally introduce this equivalence here.

The key is the concept of a \emph{normalization} of a coaction $(A,\delta)$, comprising a normal coaction $(B,\epsilon)$ together with a morphism $\phi:(A,\delta)\to (B,\epsilon)$ in \coact\ such that
\[
\phi\times G:A\times_\delta G\iso B\times_{\epsilon} G
\]
is an isomorphism.

It transpires that normalizations always exist, and are automatically surjective (see \cite[Proposition~4.5]{boiler}, or \cite[Propositions~2.3 and 2.6]{fullred}). Moreover, as we will discuss in the next lecture, normalization is a functor
\[
\nor:\coact\to\ncoact,
\]
and in fact restricts to an equivalence from \mcoact\ to \ncoact.

Anyway, the strategy for proving full Landstad duality (\thmref{full landstad category}) is to show that the diagram
\[
\xymatrix@C+30pt@R+20pt{
\act \ar[r]^-{\cp} \ar[dr]_(.3){\rcp}
&(C^*(G),\delta_G)\dn\mcoact \ar[d]^{\nor}
\\
&(C^*(G),\delta_G)\dn\ncoact
}
\]
commutes
in the sense that the composition
$\nor\circ\cp$ is naturally isomorphic to $\rcp$ (in fact, with an appropriate choice of conventions they are equal).

\subsection{Enchilada categories}

Time permitting, we take a break to give here an outline of the categorical approach to 
``classical'' crossed-product duality.
In \cite{BE} we used a different category of $C^*$-algebras, 
in which the morphisms are isomorphism classes of correspondences, 
and in which the isomorphisms are Morita equivalences. 
Part of what we showed was that
a crossed-product functor for actions of a fixed group~$G$
is  quasi-inverse to a crossed-product functor for coactions of~$G$.

\newpage

\section{Maximal-normal equivalence}
\label{lecture 5}

Recall from the preceding lecture that every coaction $(A,\delta)$ has a normalization, i.e., a normal coaction $(B,\epsilon)$ together with a morphism $\phi:(A,\delta)\to (B,\epsilon)$ in \coact\ such that
\[
\phi\times G:A\times_\delta G\to B\times_{\epsilon} G
\]
is an isomorphism.
Similarly,
a \emph{maximalization} of $(A,\d)$ is a maximal coaction $(B,\epsilon)$ together with a morphism $\phi:(B,\epsilon)\to (A,\delta)$ in \coact\ such that
\[
\phi\times G:B\times_{\epsilon} G\to A\times_\delta G
\]
is an isomorphism.
Maximalizations always exist
\cite[Theorem~3.3]{ekq}, and are automatically surjective.

The original idea was that any coaction could be replaced with a maximal one or a normal one with the same crossed product; but Fischer \cite{fischer:quantum} has subsequently shown that more is true.
Maximalizations and normalizations have the
following universal properties (for maximalizations see
\cite[Lemma~6.2]{fischer:quantum}, and for normalizations see
\cite[Lemma~4.2]{fischer:quantum} and also
\cite[Lemma~2.1]{ekq} --- the requirement
in \cite{ekq}
that the homomorphisms are into the $C^*$-algebras themselves rather
than into multipliers is not used in the proof of \cite[Lemma~2.1]{ekq}):

\begin{enumerate}
\item If $(B,\epsilon,\phi)$ is a maximalization of $(A,\delta)$, $(C,\gamma)$ is a maximal coaction, and $\tau:(C,\gamma)\to (A,\delta)$ is a morphism in \coact, then there is a unique morphism $\sigma$ making the diagram
\[
\xymatrix{
(C,\gamma) \ar[dr]^\tau \ar@{-->}[d]_\sigma^{!}
\\
(B,\epsilon) \ar[r]_\phi
&(A,\delta)
}
\]
commute in \coact.

\item If $(B,\epsilon,\phi)$ is a normalization of $(A,\delta)$, $(C,\gamma)$ is a normal coaction, and $\tau:(A,\delta)\to (C,\gamma)$ is a morphism in \coact, then there is a unique morphism $\sigma$ making the diagram
\[
\xymatrix{
(A,\delta) \ar[dr]_\tau \ar[r]^\phi
&(B,\epsilon) \ar@{-->}[d]^\sigma_{!}
\\
&(C,\gamma)
}
\]
commute.
\end{enumerate}

Thus, 
a maximalization of $(A,\delta)$ is a \emph{final object} in the comma category $\mcoact\dn (A,\delta)$, and
a normalization is an \emph{initial object} in the comma category $(A,\delta)\dn\ncoact$.

The above universal properties 
allow us to turn maximalization and normalization into functors from \coact\ into \mcoact\ and \ncoact, respectively.
To wit: suppose that a particular maximalization $(A^m,\delta^m,q_A^m)$ has been chosen for every coaction $(A,\delta)$.
Then for each morphism $\phi:(A,\delta)\to (B,\epsilon)$ in \coact\
there exists a unique morphism $\phi^m$ making the diagram
\begin{equation}
\label{maximalization functor}
\xymatrix{
(A^m,\delta^m) \ar@{-->}[r]^-{\phi^m}_-{!} \ar[d]_{q^m_A}
&(B^m,\epsilon) \ar[d]^{q^m_B}
\\
(A,\delta) \ar[r]_-\phi
&(B,\epsilon)
}
\end{equation}
commute.
Uniqueness of $\phi^m$ then ensures that this assignment is functorial; thus there is a unique functor $\maxi:\coact\to\mcoact$ given by
\[
\maxi(A,\delta)=(A^m,\delta^m)
\midtext{and}
\maxi(\phi)=\phi^m,
\]
such that 
\eqref{maximalization functor} commutes;
equivalently, such that
the assignment $A\mapsto q_A^m$ 
is a natural transformation from the composition $\inc\circ\maxi$ to the identity functor on \coact, where $\inc:\mcoact\hookrightarrow\coact$ is the inclusion functor.

Similarly, once a normalization $(A^n,\delta^n,q_A^n)$ has been chosen for every coaction $(A,\delta)$, there is a unique functor $\nor:\coact\to\ncoact$ given by
\[
\nor(A,\delta)=(A^n,\delta^n)
\midtext{and}
\nor(\phi)=\phi^n,
\]
such that for every morphism $\phi:(A,\delta)\to (B,\epsilon)$ in \coact\ the diagram
\[
\xymatrix{
(A,\delta) \ar[r]^-\phi \ar[d]_{q^n_A}
&(B,\epsilon) \ar[d]^{q^n_B}
\\
(A^n,\delta^n) \ar[r]_-{\phi^n}
&(B^n,\epsilon^n)
}
\]
commutes; equivalently,
$A\mapsto q_A^n$ is a natural transformation from 
the identity functor on \coact\ to
the composition $\inc\circ\maxi$,
where $\inc:\ncoact\hookrightarrow\coact$ is the inclusion functor.

Getting back to the universal properties:
they imply that for each coaction $(A,\delta)$ and each maximal coaction $(C,\gamma)$, the assignment
\[
\psi\mapsto q_A^m\circ\psi
\]
is a bijection:
\[
\mor_{\mcoact}((C,\gamma),(A^m,\delta^m)) \to \mor_{\coact}((C,\gamma),(A,\delta)).
\] 
Moreover, the resulting collection of bijections is \emph{natural} in $(A,\delta)$ and $(C,\gamma)$:
for any morphisms $\tau\colon (D,\eta)\to (C,\gamma)$ in $\mcoact$
and $\phi\colon(A,\delta)\to (B,\epsilon)$ in $\coact$,
we have
\[
 \phi^m\circ\psi\circ\tau
\mapsto q_B^m \circ (\phi^m\circ\psi\circ\tau)
 = \phi\circ (q_A^m \circ \psi)\circ\tau
\]
by~\eqref{maximalization functor}.

In categorical terms, what this all means is that 
$\maxi$ is a \emph{right adjoint}
for the inclusion functor $\inc^m:\mcoact\to\coact$;
\emph{i.e.}, 
the pair $(\inc^m,\maxi)$ is an \emph{adjunction}.
% 
% \mcoact\ is a \emph{coreflective} subcategory of \coact, and
% $\maxi$ is a \emph{coreflector}.

Similarly, for each coaction $(A,\delta)$ and each normal coaction $(C,\gamma)$, the assignment
\[
\psi\mapsto \psi\circ q_A^n
\]
is a bijection:
\[
\mor_{\ncoact}((A^n,\delta^n),(C,\gamma))\to\mor_{\coact}((A,\delta),(C,\gamma))
\]
which is natural in $(A,\delta)$ and $(C,\gamma)$,
so that 
$\nor$ is a \emph{left adjoint} of the inclusion functor $\inc^n:\ncoact\to\coact$,
and $(\nor,\inc^n)$ is an adjunction. 
% \item \ncoact\ is a \emph{reflective} subcategory of \coact, and
% 
% \item $\nor$ is a \emph{reflector}.

Restricting $\nor$ to \mcoact\ and $\maxi$ to \ncoact, by generalized abstract nonsense\footnote{i.e., category theory}, thus makes the restrictions 
\[
(\nor|,\maxi|) = (\nor\circ\inc^m, \maxi\circ\inc^n)
\]
an adjoint pair of functors between \mcoact\ and \ncoact.

If we knew slightly more, namely that $\nor|$ and $\maxi|$ were \emph{quasi-inverses}, so that
\[
\nor|\circ\maxi|\cong \Id_{\ncoact}
\midtext{and}
\maxi|\circ\nor|\cong \Id_{\mcoact},
\]
we would have proved the \emph{maximal-normal equivalence}:

\begin{thm}
[{\cite[Corollary~3.4]{clda}}]
The restrictions
\[
\nor|:\mcoact\to\ncoact
\quad\text{and}\quad
\maxi|:\ncoact\to\mcoact
\]
are quasi-inverses, and hence give an equivalence of categories 
between $\mcoact$ and $\ncoact$.
\end{thm}

We'll accomplish this in the next lecture.

\newpage

\section{Proofs of maximal-normal equivalence, and of Full CLDA}
\label{lecture 6}

\subsection{Maximal-normal equivalence}

In the last lecture, you heard about the equivalence between maximal and normal coactions.
Now we'll give the proof of that, and also show how the maximal-normal equivalence is used to prove categorical Landstad duality for full crossed products by actions.

Recall all the bits involved in the maximal-normal equivalence:
We have, loosely speaking, a commutative diagram of adjunctions
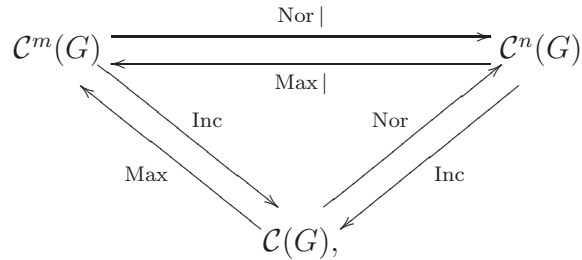
\sknote{loosely speaking\dots}
\[
\xymatrix@C+30pt@R+30pt{
\mcoact \ar@<1ex>[rr]^-{\nor|}
\ar@<1ex>[dr]^{\inc}
&&\ncoact \ar@<1ex>[ll]^-{\maxi|}
\ar@<1ex>[dl]^{\inc}
\\
&\coact \ar@<1ex>[ul]^{\maxi} \ar@<1ex>[ur]^{\nor},
}
\]
and we want to know that not only are the functors
$\nor|$ and $\maxi|$ left and right adjoints of each other (because adjunctions are preserved under composition),
but also:

\begin{thm}
\label{max-nor}
With the above notation,
$\nor|$ and $\maxi|$ are quasi-inverse equivalences.
% between the categories of maximal and normal coactions.
\end{thm}

\begin{proof}
This depends upon a certain symmetry built into the definitions of maximalization and normalization:
If $(A,\delta)$ is normal then
$q^m:(A^m,\delta^m)\to (A,\delta)$ is not only a maximalization of $(A,\delta)$, but also a normalization of $(A^m,\delta^m)$, because $q^m\times G:A^m\times_{\delta^m} G\to A\times_\delta G$ is an isomorphism,
and similarly
if $(A,\delta)$ is maximal then
$q^n:(A,\delta)\to (A^n,\delta)$ is not only a normalization of $(A,\delta)$, but also a maximalization of $(A^n,\delta^n)$.
We thus have universal properties on both sides:
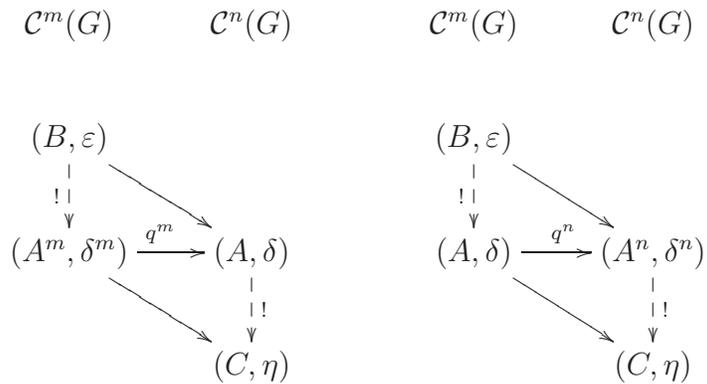
\begin{align*}
\xymatrix{
\mcoact&\ncoact
\\
(B,\epsilon) \ar[dr] \ar@{-->}[d]_{!}
\\
(A^m,\delta^m) \ar[r]^{q^m} \ar[dr]
&(A,\delta) \ar@{-->}[d]^{!}
\\
&(C,\eta)
}
&&
\xymatrix{
\mcoact&\ncoact
\\
(B,\epsilon) \ar[dr] \ar@{-->}[d]_{!}
\\
(A,\delta) \ar[r]^{q^n} \ar[dr]
&(A^n,\delta^n) \ar@{-->}[d]^{!}
\\
&(C,\eta)
}
\end{align*}

Ok, to start the proof of the theorem, we take advantage of
some abstract nonsense: it transpires that,
since $\maxi|$ is a right adjoint of $\nor|$,
\sknote{Invoking Corollary~4.3 of \emph{Adjoints for the Rest of Us},
which is non-obvious!}
it suffices to show that
\[
\nor|:\mcoact\to\ncoact
\]
is
\begin{enumerate}
\item full and faithful, and

\item essentially surjective.
\end{enumerate}
For (1), if $(A,\delta)$ and $(B,\epsilon)$ are maximal,
then any diagram of the form
\[
\xymatrix{
(A,\delta) \ar[r]^-{q^n_A} \ar@{-->}[d]_\phi^{!}
&(A^n,\delta^n) \ar[d]^{\psi=\phi^n}
\\
(B,\epsilon) \ar[r]_-{q^n_B}
&(B^n,\epsilon^n)
}
\]
can be uniquely completed commutingly,
because first of all $\phi$ exists uniquely to make the diagram commute
since $q^n_B$ is final in $\mcoact\dn (B^n,\epsilon^n)$,
and then
we must have $\psi=\phi^n$ because $\phi^n$ is the unique morphism such that $\phi^n\circ q^n=q^n\circ\phi$.

On the other hand, for (2),
if $(B,\epsilon)$ is normal,
then both morphisms
\[
q^m:(B^m,\epsilon^m)\to (B,\epsilon)
\midtext{and}
q^n:(B^m,\epsilon^m)\to (B^{mn},\epsilon^{mn})
\]
are initial in $(B^m,\epsilon^m)\dn\ncoact$,
and hence are isomorphic,
and consequently $(B,\epsilon)$ is isomorphic to the normalization
$(B^{mn},\epsilon^{mn})$ of the maximal coaction
$(B^m,\epsilon^m)$.
\end{proof}

\begin{q}
By Fischer's work (\cite{fischer:quantum}),
in \thmref{max-nor} we could just as well consider
categories of coactions of a fixed quantum group. 
What are other examples of subcategories satisfying all the relationships we needed for \thmref{max-nor}? 
We have been frustrated by our inability to find examples of this phenomenon in the literature. We have only been able to construct a couple of rather unsatisfying and somewhat artificial examples, where the subcategories are in fact isomorphic.
\end{q}

\subsection{Full CLDA}

We can now indicate how to prove the categorical Landstad duality for full crossed products.
Recall from Lecture 4 that we have an equivalence
\[
\widetilde\nor: (C^*(G),\delta_G)\dn\mcoact\to (C^*(G),\delta_G)\dn\ncoact,
\]
and the idea is to show that the diagram
\[
\xymatrix@C+30pt@R+20pt{
\act \ar[r]^-{\cp} \ar[dr]_(.3){\rcp}
&(C^*(G),\delta_G)\dn\mcoact \ar[d]^{\nor}
\\
&(C^*(G),\delta_G)\dn\ncoact
}
\]
commutes
in the sense that the composition
$\nor\circ\cp$ is naturally isomorphic to $\rcp$.

But notice: we haven't yet said anything about why we in fact have an equivalence
\[
\nor:(C^*(G),\delta_G)\dn\mcoact\sim (C^*(G),\delta_G)\dn\ncoact.
\]
Well, this can't really make sense, because $\nor$ is a functor from \mcoact\ to \ncoact\ (and in fact we've already abused the notation: it really should be the restriction $\nor|$).
So, more precisely, we must actually mean that we can produce an equivalence
\[
\wilde\nor:(C^*(G),\delta_G)\dn\mcoact\to (C^*(G),\delta_G)\dn\ncoact
\]
that somehow ``extends'' $\nor$. Here's the justification of that\footnote{and we mention here that the following result, as is the case with some others in these talks, is really just a special case of some abstract category theory; so, we could clean up the notation for the time being by stating and proving this abstract result, and then applying it to what we want, but the trade-off is that the audience would then be obliged to make all the connections between the general and the particular cases}:

\begin{prop}
If $\phi:(C^*(G),\delta_G)\to (A,\delta)$
in \mcoact,
so that $(A,\delta,\phi)$ is an object in $(C^*(G),\delta_G)\dn\mcoact$,
define
\[
\wilde\nor(A,\delta,\phi)=(A^n,\delta^n,q^n\circ\phi),
\]
an object
in $(C^*(G),\delta_G)\dn\ncoact$.

If $\sigma:(A,\delta,\phi)\to (B,\epsilon,\psi)$ in $(C^*(G),\delta_G)\dn\mcoact$,
define
\[
\wilde\nor(\sigma)=\sigma^n:\wilde\nor(A,\delta,\phi)\to \wilde\nor(B,\epsilon,\psi)
\]
in $(C^*(G),\delta_G)\dn\ncoact$.

Then $\wilde\nor$ is an equivalence.
\end{prop}

\sknote{Of course, this proposition and its proof are just translated from the abstract categorical version in \cite[Corollary~2.2]{clda}.  Now I'm not so sure I stand by the idea of 
not doing any abstract categorical nonsense\dots}

Note that we can regard $\wilde\nor$ as extending $\nor$ in the sense that $\nor(A,\delta)=(A^n,\delta^n)$.

\begin{proof}
We must show that $\wilde\nor$ is full, faithful, and essentially surjective.
For the first two properties,
let
\[
\rho:\wilde\nor(A,\delta,\phi)\to \wilde\nor(B,\epsilon,\psi).
\]
Consider the diagram
\[
\xymatrix{
&(C^*(G),\delta_G)
\ar[dl]_\phi \ar[dr]^\psi
\ar'[d]^{q_G^n}[dd]
\\
(A,\delta)
\ar[dd]_{q_A^n}
\ar@{-->}[rr]_(.3)\sigma^(.3){!}
&&(B,\epsilon)
\ar[dd]^{q_B^n}
\\
&(C^*_r(G),\delta_G^n)
\ar[dl]_{\phi^n} \ar[dr]^{\psi^n}
\\
(A^n,\delta^n)
\ar[rr]_-{\rho=\sigma^n}
&&(B^n,\epsilon^n).
}
\]
% \[
% \xymatrix{
% &a
% \ar[dd]|!{[dl];[dr]}\hole
% \\
% b
% &&c
% \\
% &d
% }
% \]
We want a unique $\sigma$ for which the top triangle commutes and $\rho=\wilde\nor\sigma=\sigma^n$.
Since $\nor:\mcoact\to\ncoact$ is full and faithful,
there is a unique $\sigma$ such that $\rho=\sigma^n$,
and then it suffices to show that the top triangle commutes.
For this, since $\nor$ is faithful,
it suffices to show that the bottom triangle commutes.
We have
\[
\rho\circ\phi^n\circ q^n
=\rho\circ q^n\circ\phi
=q^n\circ\psi
\]
(because $\rho$ is a morphism in the comma category),
and since $\psi^n$ is the unique morphism such that $\psi^n\circ q^n=q^n\circ\psi$, we must have
\[
\rho\circ\phi^n=\psi^n.
\]

For the essential surjectivity,
let $\phi:(C^*(G),\delta_G)\to (A,\delta)$
in \coact\ with $\delta$ normal,
and consider the diagram
\[
\xymatrix{
(C^*(G),\delta_G)
\ar[dr]_\phi
\ar@{-->}[r]^-\psi
&(A^m,\delta^m)
\ar[d]^{q^m} \ar[dr]^{q^n}
\\
&(A,\delta)
\ar@{-->}[r]_-\theta^-\cong
&(A^{mn},\delta^{mn}).
}
\]
Since $(A^m,\delta^m,q^m)$ is final in $\mcoact\dn (A,\delta)$,
there exists $\psi$ making the left triangle commute.
Since both $(A,\delta,q^m)$ and $(A^{mn},\delta^{mn},q^n)$
are initial in $(A^m,\delta^m)\dn\ncoact$,
there exists an isomorphism $\theta$ making the right triangle commute.
Then
\[
\theta:(A,\delta,\phi)\iso \widetilde\nor(A^m,\delta^m,\psi).
\qedhere
\]
\end{proof}

OK, that gives the equivalence
\[
(C^*(G),\delta_G)\dn\mcoact \sim
(C^*(G),\delta_G)\dn\ncoact.
\]
But
we still need to see how this allows us to deduce categorical Landstad duality for full crossed products from the version for reduced crossed products.
Well, it turns out that we can cheat:
we need to see that $\rcp$ is naturally isomorphic to $\nor\circ\cp$, but in fact we can choose our conventions so that they coincide, and that does it.
How can we make those choices?
By taking the normalization $A^n$ of $A$ to be the quotient of $A$ by $\ker j_A$, and the reduced crossed product $A\times_{\alpha,r} G$ to be the quotient of the full crossed product $A\times_\alpha G$ by the kernel of the regular representation.

\begin{rem}
The categorical perspective we have described in these lectures serves several purposes: 
first of all, it allows for better results. 
More precisely, there are some places where we were only able to prove the full extent 
of what we wanted once we took the categorical approach. 
For example, in the Landstad duality for full crossed products, 
we were only able to prove the uniqueness clause as we stated it in Lecture 4 
by using the universal properties in \cite{fischer:quantum}.

Taking a wider view, however, we want to suggest that what we have presented here 
is a method that can be profitably used in other circumstances. 
For example, 
in the context of classical crossed-product duality, 
the categorical perspective taken in~\cite{BE} highlighted various lacunae in the existing theory which we 
were subsequently able to fill.
\end{rem}

\newpage

%----------------------------------------------------------------

%\bibliographystyle{amsplain}
%\bibliography{cstar}

\providecommand{\bysame}{\leavevmode\hbox to3em{\hrulefill}\thinspace}
\providecommand{\MR}{\relax\ifhmode\unskip\space\fi MR }
\providecommand{\MRhref}[2]{%
  \href{http://www.ams.org/mathscinet-getitem?mr=#1}{#2}
}
\providecommand{\href}[2]{#2}

\end{document}